\documentclass[onecolumn]{article}

 \usepackage{graphicx}          
 \usepackage{amssymb}
 \usepackage{amsmath}
\usepackage{amsthm}

\usepackage{titling}

\newtheorem{thm}{Theorem}

\theoremstyle{definition}
\usepackage{thmtools}

\theoremstyle{remark}

\theoremstyle{definition}

\title{A note on norm-based Lyapunov functions via contraction analysis}
\author{Samuel Coogan and Murat Arcak\thanks{The authors are with the Department of Electrical Engineering and Computer Sciences, University of California, Berkeley, Berkeley, CA. \texttt{\{scoogan,arcak\}@eecs.berkeley.edu}}}
\date{}

\renewcommand\footnotemark{}
\begin{document}
\maketitle

\abstract{It is well know that for globally contractive autonomous systems, there exists a unique equilibrium and the distance to the equilibrium evaluated along any trajectory decreases exponentially with time. We show that, additionally, the magnitude of the velocity evaluated along any trajectory decreases exponentially, thus giving an alternative choice of Lyapunov function.}

\section{Main Result}

Consider the nonlinear system
\begin{align}
  \label{eq:1}
  \dot{x}=f(x)
\end{align}
for $x\in\mathbb{R}^n$ and continuously differentiable $f(\cdot)$. Denote the Jacobian as
\begin{align}
  \label{eq:2}
  J(x)\triangleq \frac{\partial f}{\partial x}(x).
\end{align}

Let $|\cdot|$ be a vector norm on $\mathbb{R}^n$, $\Vert\cdot\Vert$ its induced matrix norm, and $\mu(A)\triangleq \lim_{h\to 0^+}\frac{1}{h}(\Vert I+hA\Vert-1)$ the associated matrix measure.
\begin{thm}
\label{thm:1}
If there exists $c>0$ such that $\mu(J(x))\leq -c$ for all $x\in \mathbb{R}^n$, then 
  \begin{align}
    \label{eq:3}
    |f(x(t))|\leq |f(x(0))|e^{-ct}.
  \end{align}
\end{thm}
  \begin{proof}
    Let $V(x)\triangleq |f(x)|$. $V(x(t))$ is then absolutely continuous as a function of $t$ and therefore
    \begin{align}
      \label{eq:4}
      \dot{V}(x(t))&\triangleq \lim_{h\to 0+}\frac{V(x(t+h))-V(x(t))}{h}
    \end{align}
for almost all $t$. Furthermore,
\begin{align}
  \label{eq:8}
\lim_{h\to 0^+}\bigg|\frac{|f(x(t+h)))|-|f(x)+h\dot{f}(x)|}{h}\bigg|&\leq \lim_{h\to 0^+}\left|\frac{f(x(t+h))-f(x)}{h}-\dot{f}(x)\right|\\
\label{eq:8-2}&=0
\end{align}
where we use the definition of $\dot{f}(x)$ and the fact $\big||x|-|y|\big|\leq |x-y|$. Since also $\dot{f}(x)=J(x)f(x)$, we combine \eqref{eq:4}--\eqref{eq:8-2} and obtain
\begin{align}
  \label{eq:6}
  \dot{V}(x(t))&=\lim_{h\to 0+}\frac{|f(x)+hJ(x)f(x)|-|f(x)|}{h}\\
&\leq\lim_{h\to 0+} \frac{\Vert I+hJ(x)\Vert \cdot |f(x)|-|f(x)|}{h}\\
&=\lim_{h\to 0+} \frac{\Vert I+hJ(x)\Vert-1}{h}|f(x)|\\
&=\mu(J(x))V(x).
\end{align}
By hypothesis, we then have $\dot{V}(x)\leq -c V(x)$, and \eqref{eq:3} follows by integration. 
  \end{proof}

Furthermore, under the hypothesis of Theorem \ref{thm:1}, $f(x)$ is a diffeomorphism from $\mathbb{R}^n$ to itself \cite[pp. 34--35]{Desoer:2008bh}, thus $V(x)$ is positive definite and radially unbounded. By standard Lyapunov theory \cite{khalil}, global asymptotic stability of the unique equilibrium $x^*\triangleq f^{-1}(0)$ follows.

\section{Discussion}
  It is well known that for contractive systems with $\mu(J(x))\leq -c$ for all $x$, $|x(t)-\xi(t)|\leq |x(0)-\xi(0)|e^{-ct}$ for any pair of trajectories $x(\cdot)$ and $\xi(\cdot)$ \cite{Sontag:2010fk}. Taking $\xi(t)\equiv x^*\triangleq f^{-1}(0)$, we see that for such systems, $V(x)=|x-x^*|$ is a Lyapunov function guaranteeing global asymptotic stability. Theorem \ref{thm:1} demonstrates that $V(x)=|f(x)|$ is an alternative choice of Lyapunov function.  Additionally, it is straightfoward that for any continuously differentiable class-$\mathcal{K}$ function $\rho(\cdot)$,
\begin{align}
  \label{eq:11}
  x\mapsto \rho(|f(x)|)
\end{align}
is also a Lyapunov function.

Furthermore, Theorem \ref{thm:1} is a generalization of Krasovskii's well known sufficient condition for asymptotic stability \cite{Krasovskii:1963ve}. Indeed, Krasovskii's theorem is a special case of Theorem \ref{thm:1} when $|\cdot|$ is taken to be a weighted Euclidean norm, \emph{i.e.} $|x|=(x^TPx)^{1/2}$ for some positive definite matrix $P$ and therefore \cite{Desoer:2008bh}
\begin{align}
  \label{eq:7}
  \mu(A)=\bar{\lambda}_i\left(\frac{B+B^T}{2}\right),\quad B\triangleq P^{1/2}AP^{-1/2}
\end{align}
where $\bar{\lambda}_i(M)$ denotes the largest eigenvalue of $M$, from which it follows
\begin{align}
  \label{eq:10}
  \mu(A)<-c \quad \iff \quad PA+A^TP<-2cP.
\end{align}
Using Theorem \ref{thm:1} with \eqref{eq:10} and taking $\rho(y)\triangleq y^2$ in \eqref{eq:11} gives a standard version of Krasovskii's theorem.

\bibliographystyle{ieeetr}
\bibliography{$HOME/Documents/Books/books}

\begin{thebibliography}{1}

\bibitem{Desoer:2008bh}
C.~Desoer and M.~Vidyasagar, {\em Feedback systems: Input-output properties}.
\newblock Society for Industrial and Applied Mathematics, 2008.

\bibitem{khalil}
H.~K. Khalil, {\em Nonlinear Systems}.
\newblock Prentice Hall, third~ed., 2002.

\bibitem{Sontag:2010fk}
E.~D. Sontag, ``Contractive systems with inputs,'' in {\em Perspectives in
  Mathematical System Theory, Control, and Signal Processing}, pp.~217--228,
  Springer, 2010.

\bibitem{Krasovskii:1963ve}
N.~Krasovskii, {\em Stability of Motion}.
\newblock Stanford University Press, 1963.

\end{thebibliography}
\end{document}